\documentclass[12pt]{article}
\usepackage{url}
\usepackage{fancyhdr}
\usepackage{amsmath, amssymb, latexsym, amscd, amsthm,amsfonts,amstext}
\usepackage{subfigure}
\usepackage{enumerate}
\usepackage{xcolor}
\usepackage{textcomp}
\usepackage{mathtools}

\usepackage{graphicx}
\newtheorem{theorem}{Theorem}[section]
\newtheorem{lemma}[theorem]{Lemma}

\numberwithin{equation}{section}
\newtheorem{definition}[theorem]{Definition}
\newtheorem{remark}[theorem]{Remark}

\numberwithin{equation}{section}

\usepackage[mathscr]{eucal}
\usepackage{graphicx}
\usepackage{subfigure}

\usepackage{fancyhdr}

\usepackage{enumerate}
\usepackage{xcolor}

\allowdisplaybreaks

\title{A cluster of many small holes with negative imaginary surface impedances may generate 
a negative refraction index} 
\author{
Ahmed Alsaedi
 \thanks{Nonlinear Analysis and Applied Mathematics Research Group (NAAM), Department of  Mathematics,
 Faculty of Sciences,
 King Abdulaziz University,
 P.O. Box 80203,
 Jeddah 21589,
 Saudi Arabia.}
 \and
 Bashir Ahmad \footnotemark[1]
 \and  
  Durga Prasad Challa \footnote{C\lowercase{orresponding author}: D\lowercase{urga} Prasad C\lowercase{halla}.}
 \thanks{Department of mathematics, 
 Tallinn University of Technology, Tallinn 19068, Estonia.
 (Email: durga.challa@ttu.ee).
 }
 \and Mokhtar Kirane\footnotemark[1]
 \thanks{Laboratoire de Math\'ematiques,
 P\^ole Sciences et Technologies, 
 Universit\'e de La Rochelle, 
 Avenue Michel Cr\'epeau
 17042, La Rochelle Cedex, France, (Email: mokhtar.kirane@univ-lr.fr).}
 \and  
 Mourad Sini \thanks{RICAM, Austrian Academy of Sciences,
 Altenbergerstrasse 69, A-4040, Linz, Austria.
 (Email: mourad.sini@oeaw.ac.at).
 }
}
 
\begin{document}
% \graphicspath{{Figures-eps/}}
 \maketitle

\begin{abstract}
We deal with the scattering of an acoustic medium modeled by an index of refraction $n$ varying in a bounded 
region $\Omega$ of $\mathbb{R}^3$ and equal to unity outside $\Omega$. 
This region is perforated with an extremely large number of small holes $D_m$'s of 
maximum radius $a$, $a<<1$, modeled by surface impedance functions. 
Precisely, we are in the regime described 
by the number of holes of the order $M:=O(a^{\beta-2})$, the minimum distance between the holes is $d\sim a^t$ and the surface impedance functions of the form 
$\lambda_m \sim \lambda_{m,0} a^{-\beta}$ with $\beta >0$ and $\lambda_{m,0}$ being constants and 
eventually complex numbers. 
Under some natural conditions on the parameters 
$\beta, t$ and $\lambda_{m,0}$, we characterize the equivalent medium 
generating, approximately, the same scattered waves as the original perforated acoustic medium. We give an explicit error estimate between the scattered waves generated by the perforated
medium and the equivalent one respectively, as $a \rightarrow 0$. As applications of these results, we discuss the following findings:
\begin{enumerate}
\item If we choose negative valued imaginary surface impedance functions, 
 attached to each surface of the holes, then the equivalent medium behaves as a passive 
 acoustic medium only if it is an acoustic metamaterial with index of refraction
 $\tilde{n}(x)=-n(x),\; x \in \Omega$ and $\tilde{n}(x)=1,\; x \in \mathbb{R}^3\setminus{\overline{\Omega}}$. 
 This means that, with this process, we can switch the sign of the index of the refraction from positive to negative values.

 \item We can choose the surface impedance functions attached to each surface of the holes so that the equivalent index of 
 refraction $\tilde{n}$ is $\tilde{n}(x)=1,\; x \in \mathbb{R}^3$. 
 This means that the region $\Omega$ modeled by the original index of refraction $n$ is 
 approximately cloaked.
 
 \end{enumerate}
 \end{abstract}

%\begin{keywords}
%Acoustic scattering, Small-scatterers, negative refraction index, metamaterials.
%\end{keywords}

%\begin{AMS}
% 35J08, 35Q61, 45Q05 
%\end{AMS}

\pagestyle{myheadings}
 \thispagestyle{plain}
 %\markboth{A. Alsaedi D. P. Challa. M. Kirane and M. Sini }{A cluster of many small bodies of impedance type generate a negative refraction index}

\section{Introduction}\label{Introduction-smallac-sdlp}

%\subsection{A general introduction}
 The derivation of the macroscopic behavior of a given physical system (or other biological, chemical systems, etc.) as an interaction of a 'dense' microscopic particles
is a well known procedure for a long time, see \cite{Choy:1999, Barnard:2010}. However, the 
 mathematical modeling as well as
 the justification of this procedure was understood only in the middle of the 
 last century, see \cite{B-L-P:1978, J-K-O:1994, M-K:2006}. One of the most 
 known ideas to describe the passage from
 the microscopic states to the equivalent macroscopic state is the 
 homogenization theory, see \cite{B-L-P:1978, J-K-O:1994, M-K:2006}. There are two approaches: the deterministic
 one and the probabilistic one. If one wants to estimate the macroscopic state 
 deterministically, then one needs to assume the periodicity in distributing the 
 small particles. To avoid the periodicity, one can assume the small particles
 to be randomly distributed and in this case we estimate the equivalent medium in 
 the probabilistic sense. Let us emphasize here that in both approaches, 
 we estimate the limit of the fields created by the microscopic structure to 
 the fields created by the macroscopic one with energy norms stated in the whole domain where
 the small particles are distributed.
 
 If we are interested only in estimating the limit of the fields away from 
 these small particles (as in the inverse problems and the design theories), then another alternative, avoiding both the periodicity and the randomness,
 is possible.
 
 A root of this alternative goes back to the seminal work by L. Foldy, see \cite{LLF:PR1945}, where he gave a close form of the field scattered by 
 $M$ isotropic point-like scatterers, see \cite{LLF:PR1945, Martin:2006} for details on this model.
 The justification, or the mathematical foundation, of Foldy's representation
 was proposed by Berezin and Faddeev in another seminal work, 
 see \cite{Berezin-Faddeev:1961}. The idea is that based on the 
 Krein extension theory of self-adjoint operators, one can model the diffusion by 
 point-like particles by the Schroedinger model with singular potentials of 
 Dirac type supported on those point-like scatterers. This has opened a very 
 fruitful direction of research, see the book \cite{A-G-H-H:AMS2005}, on exact models.
 Following Fadeev's approach, using Krein's extension theory, the diffusion by 
 small particles was studied in several works, see for instance \cite{Figari:1985, Na-So:2006}.

 A different, but still related, approach to describe the diffusion by small 
 particles is based on the integral equations. This is suggested by different 
 authors, as A. Ramm \cite{Ramm:2005} and
 H. Ammari and H. Kang \cite{A-K:2007}, for instance. 
 In particular, A. Ramm \cite{Ramm:2005, Ramm:2007} shows
 that the dominant term in 
 the expansion of the scattered field has the same form as the Foldy's close form
 where the centers of the small scatterers play the role of the point-like 
 particles. He used the (rough) condition 
$\frac{a}{d}\ll 1$, where $a$ is the maximum radius of the small particles and $d$ the 
minimum distance between them, and no error estimate is derived, but he formally characterized 
the equivalent medium.
We also cite the recent works by V. Maz'ya and A. Movchan 
\cite{M-M:2010, M-M-N:2013} where they study the Poisson problem and
 obtain error estimates. In their analysis, they rely on the maximum principle to 
 extend the boundary estimates, which argument does not go smoothly for stationary models 
 as the Helmholtz one.

 A rigorous approximation of the scattered fields, with error estimates in terms of the 
 number $M$ of the small scatterers, their minimum distance $d$ and the maximum radius $a$, was derived in 
 \cite{C-S:2014, C-S:2016} using the integral equations approach. 
 Based on these last estimates, we can characterize the equivalent medium with 
 explicit error estimates in terms of the parameter $a$, $a<<1$, in appropriate regimes 
 described by the other parameters $M$ and $d$ in terms of $a$, namely 
 $M:=M(a):=O(a^{-s})$ and $d:=d(a)\approx a^t$, as $a<<1$, with non negative parameters
 $s$ and $t$.
 This was done in \cite{C-S:2015, F-D-M:2016JOE} where the small particles are taken to be 
 soft acoustic or rigid elastic particles.
 The objective of the present work is to extend this study to the case where
 the small particles have impedance type surfaces with the scaled surface impedences of the form
 $\lambda_m \approx a^{-\beta}$ with non negative $\beta$, in addition to the scaled coefficients $M$ and 
 $d$ described above. Compared to the works in \cite{C-S:2015},
 we derive here better error estimates. Precisely, fixing $s=1$,  $\beta=0$ and $t=\frac{1}{3}$ 
 for simplicity and as an example, we derive here an error of the
 form $O(a^{\min\{\gamma,\; \frac{2}{3}\}})$ while in \cite{C-S:2015}
 it  is of the form $O(a^{\min\{\gamma,\;\frac{1}{15}\}})$. 
 Here $\gamma \in (0, 1]$ is the 
 Holder regularity exponent of the coefficients $\lambda_0$ and $K$ appearing 
 in the equivalent medium, as discussed below.
 
 The design of materials with desired, and in particular negative, index of refraction is a hot topic in
 the last years, see \cite{Cai-Shalaev:2010} for instance. Concerning this topic, our contribution is to have 
 shown mathematically that this is possible with high generality. Indeed, we show that 
 the equivalent medium  is modeled by three coefficients $K,P_0$ and $\lambda_0$ modeling
 respectively, the local distribution of the particles, their geometry and 
 the impedance coefficient attached to each particle. Since we have the 
 freedom in choosing the three functions $K,P_0$ and $\lambda_0$, then we can generate
 a large family of indices of refraction. In  particular 
 
 \begin{enumerate}
  \item if we choose the surface impedance
 to have negative imaginary parts, which is mathematically possible, see \cite{C-S:2016}, 
 then we show that  the 
 equivalent medium will be passive only if the index of refraction is negative, i.e.
 it behaves as a metamaterial.
 
 \item we can choose the coefficients $K,P_0$ and $\lambda_0$ so that the 
 equivalent index of refraction is educed to the unity.
 This means that the domain $\Omega$ modeled by the original 
 index of refraction $n$ is cloaked. 
 \end{enumerate}
 
 Let us emphasize that the derived explicit error estimates 
 between the fields generated by the microscopic structure and the one generated by
 the equivalent macroscopic structure might be useful to quantify
 the accuracy of the design.
 \bigskip
 
 Let us also cite some related works on the derivation of the 
 equivalent media. The first works go at least to Rauch-Taylor, 
 see \cite{Rauch-Taylor:1975}, see also the works by Cioranescu 
and Murat \cite{C-M:1979, C-M:1997}, who characterized the limiting 
problems for some Poisson type problems and provided convergence results 
(but with no rates of convergence) of the corresponding resolvent operators. 
Later, these results were extended and refined in the works of Ozawa and the ones of Figari et al., see respectively \cite{Ozawa:1983} and \cite{Figari:1985} for instance, 
using point interaction approximations of the Green's kernels. Compared to these results,
we do not need the periodicity nor the randomness in distributing the small 
particles in addition we model them via the 
scaled parameters $M$, $d$ and the $\lambda_m$'s with a high generality 
as described above.    
 \bigskip

The rest of the paper is organized as follows. In
section \ref{statement-results}, we state the main results. Precisely, 
we describe the mathematical model of the stationary scattering by 
many small bodies of impedance type in section \ref{The model-for-many-imp-holes}, then we state the main 
mathematical results in section \ref{results-equivalent-model}. 
We end this section with a discussion on the possible applications 
of these results in the acoustic metamaterials and the acoustic cloaking in 
section \ref{applications}. Section \ref{Proof-Theorem} is devoted to the proof of the main theorem of the paper.

\section{Statement of the results}\label{statement-results}

\subsection{The acoustic scattering by many impedance type holes}
\label{The model-for-many-imp-holes}
 Let $B_1, B_2,\dots, B_M$ be $M$ open, bounded and simply connected sets in $\mathbb{R}^3$ with Lipschitz boundaries
containing the origin. We assume that the Lipschitz constants of $B_j$, $j=1,..., M$ are uniformly bounded.  
We set $D_m:=\epsilon B_m+z_m$ to be the small bodies characterized by the parameter 
$\epsilon>0$ and the locations $z_m\in \mathbb{R}^3$, $m=1,\dots,M$. Let $U^{i}$ be a solution of the Helmholtz equation $(\Delta + \kappa^{2})U^{i}=0 \mbox{ in } \mathbb{R}^{3}$.  
We denote by  $U^{s}$ the acoustic field scattered by the $M$ small bodies $D_m\subset \mathbb{R}^{3}$, due to 
the incident field $U^{i}$ (mainly the plane incident waves $U^{i}(x,\theta):=e^{ikx\cdot\theta}$
with the incident direction $\theta \in \mathbb{S}^2$, where $\mathbb{S}^2$ being the unit sphere), with impedance boundary conditions. Hence the total field $U^{t}:=U^{i}+U^{s}$ satisfies the 
following exterior impedance problem of the acoustic waves

\begin{equation}
(\Delta + \kappa^{2}n^2(x))U^{t}=0 \mbox{ in }\mathbb{R}^{3}\backslash \left(\mathop{\cup}_{m=1}^M \bar{D}_m\right),\label{acimpoenetrable-1}
\end{equation}
\begin{equation}
\left.\frac{\partial U^{t}}{\partial \nu_m}+\lambda_m U^{t}\right|_{\partial D_m}=0,\, 1\leq m \leq M, \label{acgoverningsupport-1}  
\end{equation}
\begin{equation}
\frac{\partial U^{s}}{\partial |x|}-i\kappa U^{s}=o\left(\frac{1}{|x|}\right), |x|\rightarrow\infty, \label{radiationc-1}
\end{equation}
Again, the scattering problem 
(\ref{acimpoenetrable-1}-\ref{radiationc-1}) is well posed in the H\"{o}lder or Sobolev spaces, see \cite{C-K:1983, C-K:1998, Mclean:2000}
 in the case $\Im \lambda_m >0$. As we said for (\ref{acimpoenetrable-1}-\ref{radiationc-1}), 
 this last condition can be relaxed to allow $\Im \lambda_m$ to be negative, see \cite{C-S:2016}. Applying Green's formula to $U^s$, we can show that
 the scattered field $U^s(x, \theta)$ has the following asymptotic expansion:
\begin{equation}\label{far-field-n}
 U^s(x, \theta)=\frac{e^{i \kappa |x|}}{4\pi|x|}U^{\infty}(\hat{x}, \theta) + O(|x|^{-2}), \quad |x|
\rightarrow \infty,
\end{equation}
where the function
$U^{\infty}(\hat{x}, \theta)$ for $(\hat{x}, \theta)\in \mathbb{S}^{2} \times\mathbb{S}^{2}$  is the corresponding far-field pattern.

\begin{definition}\label{Def1}
We define  $
a:=\max\limits_{1\leq m\leq M } diam (D_m)$,\;  $
d:=\min\limits_{\substack{m\neq j\\1\leq m,j\leq M }} d_{mj}$
where $\,d_{mj}:=dist(D_m, D_j)$ and set $\kappa_{\max}$ as the upper bound of the used wave numbers, i.e. $\kappa\in[0,\,\kappa_{\max}]$.
 The distribution of the scatterers is modeled as follows.
\begin{enumerate}
\item The number $M~:=~M(a)~:=~O(a^{-s})\leq M_{max} a^{-s}$ with a given positive constant $M_{max}$.
\item The minimum distance $d~:=~d(a)~\approx ~a^t$, i.e. $d_{min} a^t \leq d(a) \leq d_{max}a^t $, with 
given positive constants $d_{min}$ and $d_{max}$.
 \item The surface impedance $\lambda_m~:=~\lambda_{m,0}a^{-\beta}$, where $\lambda_{m,0} \neq 0$ and might be a complex number.
\end{enumerate}
\end{definition}
Here the real numbers $s$, $t$ and $\beta$ are assumed to be non negative.

\bigskip

We call the upper bounds of the Lipschitz character of $B_m$'s, $M_{max}, d_{min}, d_{max}$ and $\kappa_{max}$ 
the set of the apriori bounds. In \cite{C-S:2016}, we have shown that there exist a positive constant $a_0$, 
$\lambda_-$ and $\lambda_+$ depending only on the set of the apriori bounds and on $n_{max}$ such that
if \begin{equation} \label{conditions}
a \leq a_0,\; \vert \lambda_{m, 0}\vert \leq \lambda_+, \; \vert \Re(\lambda_{m, 0}) \vert \geq \lambda_-,\; ~~ \beta<1, \;~~ s \leq 2-\beta,\;~~\frac{s}{3}\leq t
% \,,  \;~~ \min_{j\neq m}\cos(\kappa \vert z_j- z_m\vert)\geq 0
\end{equation}
then the far-field pattern $U^\infty(\hat{x},\theta)$ has the following asymptotic expansion
\begin{equation}\label{x oustdie1 D_m farmain-1}
U^\infty(\hat{x},\theta)=V^\infty_n(\hat{x},\theta) + \sum_{m=1}^{M}V_n^t(-\hat{x},z_m)\mathbb{Q}_m+
O\left(
 a^{3-s-2\beta}\right),
 \end{equation}
uniformly in $\hat{x}$ and $\theta$ in $\mathbb{S}^2$. The constant appearing in the estimate $O(.)$ depends only on 
the set of the apriori bounds, $\lambda_-$, $\lambda_+$ and on $n_{max}$. 
The quantity\;  $V_n^t(z_m, -\hat{x})$ is the total field 
evaluated at the point $z_m$ in the direction $-\hat{x}$, corresponding to the scattering problem 
\begin{equation}
(\Delta + \kappa^{2}n^2(x))V_n^{t}=0 \mbox{ in }\mathbb{R}^{3},
\label{acimpoenetrable-2}
\end{equation}
\begin{equation}
\frac{\partial V_n^{s}}{\partial |x|}-i\kappa V_n^{s}=o\left(\frac{1}{|x|}\right),
|x|\rightarrow\infty, \label{radiationc-2}
\end{equation}
i.e. $V_n^t(z_m, -\hat{x}):=
e^{-i\hat{x} \cdot z_m}+ V_n^s(z_m, -\hat{x})$, where $V_n^s$ is the scattered field.
 The coefficients $\mathbb{Q}_m$, $m=1,..., M,$ are the solutions of the following linear algebraic system
\begin{eqnarray}\label{fracqcfracmain-1}
 \mathbb{Q}_m +\sum_{\substack{j=1 \\ j\neq m}}^{M} C_m 
 G_\kappa(z_m,z_j)\mathbb{Q}_j&=&-C_m V_n^{t}(z_m, \theta),~~
\end{eqnarray}
for $ m=1,..., M$ where $C_m:=-\lambda_m\; \vert \partial D_m \vert$. Here $G_{\kappa}(x, z)$ is the outgoing Green's function 
corresponding to the scattering problem (\ref{acimpoenetrable-2}-\ref{radiationc-2}).
  \par
The algebraic system \eqref{fracqcfracmain-1} is invertible under the condition:
%\footnote{As in the homogeneous background case, if the real parts of $\lambda_{m,\; 0}$ are negative then we can get rid of the condition $s \leq 2 -\beta$ and replace it by the following one $\min_{j\neq m}\cos(\kappa \vert z_j- z_m\vert)\geq 0$.}

\begin{eqnarray}\label{invertibilityconditionsmainthm-1}
 %t\leq 2-\beta \text{ and } 
 s \leq 2 -\beta.
\end{eqnarray} 

\subsection{The equivalent model}\label{results-equivalent-model}

As the diameter $a$ tends to zero the error term in (\ref{x oustdie1 D_m farmain-1}) tends to zero for $t$ and $s$ such that %the case 
\begin{equation}\label{general-condion-s-t}
\beta<1, \;~~ s \leq 2-\beta,\;~~\frac{s}{3}\leq t,
\end{equation}
and it is at least of the order $O(a^{1-\beta})$. Observe that we have the upper bound
\begin{equation}
 \vert \sum_{m=1}^{M}e^{-i\kappa\hat{x}\cdot z_m}Q_m\vert \leq M\sup_{m=1, ..., M}\vert Q_m\vert=O(a^{2-\beta-s}) 
\end{equation}
since $Q_m\approx \vert \lambda_m \vert \vert D_m \vert \approx a^{2-\beta}$, see \cite{C-S:2016}. Hence if the number of holes is $M:=M(a):=O(a^{-s}), \; s<2 -\beta$ and $t$ satisfies (\ref{general-condion-s-t}), $a\rightarrow 0$, 
then from (\ref{x oustdie1 D_m farmain-1}), we deduce that
\begin{equation}\label{s-smaller-1}
 U^\infty(\hat{x},\theta)\rightarrow V_n^\infty(\hat{x},\theta), \mbox{ as } a \rightarrow 0, \mbox{ uniformly in terms of } \theta \mbox{ and } \hat{x} \mbox{ in } \mathbb{S}^2.
\end{equation}
This means that this collection of holes has no effect on the homogeneous medium as $a \rightarrow 0$. 
The main concern of this work is to consider the case when $s=2-\beta$. 
Let $\Omega$ be a bounded domain, say of unit volume, 
containing the holes $D_m, m=1, ..., M$. We divide $\Omega$ 
into $[a^{\beta -2}]$ sub-domains $\Omega_m,\; m=1, ..., [a^{\beta -2}]$ 
such that each $\Omega_m$ contains $D_m$, with $z_m \in \Omega_m$ as its center,
and some of the other $D_j$'s. We assume the number of holes in $\Omega_m$, for $m=1, ..., [a^{\beta -2}]$,
to be uniformly bounded in terms of $m$. 
To be precise, we introduce $K: \mathbb{R}^3\rightarrow \mathbb{R}$ as a positive continuous and 
bounded function. Let each $\Omega_m$, $m\in \mathbb{N}$, be a cube such that $\Omega_m \cap \Omega$ (which we denote also by $\Omega_m$) 
is of volume $a^{2-\beta}\frac{[K(z_m)+1]}{K(z_m)+1}$ and contains $[K(z_m) +1]$ holes (where $[a]$ stands for the integral
part of 
$a\in \mathbb{R}$). We set $K_{max}:=\sup_{z_m}(K(z_m)+1)$, hence 
$M=\sum^{[a^{\beta -2}]}_{j=1}[K(z_m)+1]\leq K_{max}[a^{\beta-2}]=O(a^{\beta -2})$. 

\begin{figure}[htp]\label{distribution-obstacles}
\centering
\includegraphics[width=6cm,height =5cm,natwidth=610,natheight=642]
{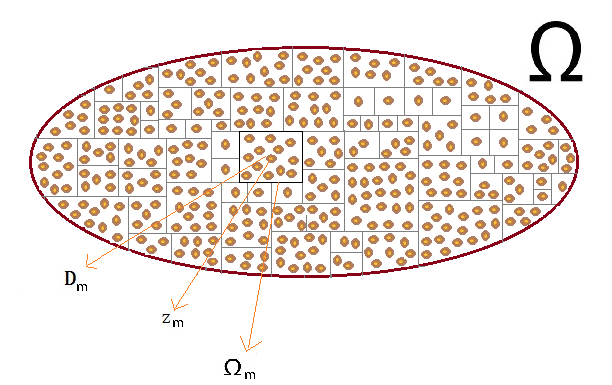}
\caption{An example on how the holes are distributed in $\Omega$.}
\end{figure}

% \begin{figure}[htp]\label{distribution-obstacles}
% \centering
% \includegraphics[width=6cm,height =5cm,natwidth=610,natheight=642]{Effectivemedium-withoutoutersketch.png}
% \caption{An example on how the obstacles are distributed in $\Omega$.}
% \end{figure}

\noindent We prove the following result.
\begin{theorem}\label{equivalent-medium}
Let $\lambda_0: \Omega \longmapsto \mathbb{C}$ be a continuous function  and take $\lambda_m:=\lambda_0(z_m) a^{-\beta}$.
Consider the small holes to be distributed, as described above, in a bounded domain $\Omega$, say of unit volume, with their number $M:=M(a):=O(a^{\beta -2})$
and their minimum distance 
$d:=d(a):=a^{t}$,\; $ \frac{2-\beta}{3}\leq t \leq 2 -\beta$ with $\beta<1$, as 
$a\rightarrow 0$.
\begin{enumerate}
 \item  If the shapes of the holes are different, under the condition that the reference 
 bodies $B_m$'s have uniformly upper bounded perimeters with uniformly lower 
 bounded radiis, then there exist a function
 $\bold{P}_0$ in $ \cap_{p\geq 1}L^p(\mathbb{R}^{3})$ with support in $\Omega$ such that
 \begin{equation}\label{B}
  \lim_{a\rightarrow 0}U^\infty(\hat{x},\theta)= U_{0}^\infty(\hat{x},\theta) \mbox{ uniformly in terms of } \theta \mbox{ and } \hat{x} \mbox{ in } \mathbb{S}^2
 \end{equation}
where $U_{0}^\infty(\hat{x},\theta)$ is the far-field corresponding to the scattering problem

\begin{equation}
(\Delta + \kappa^{2}n^2+(K+1)\bold{P}_0\lambda_0)U_{0}^{t}=0 \mbox{ in }\mathbb{R}^{3},\label{B-1}
\end{equation}
\begin{equation}
U_{0}^{t}=U_{0}^s +e^{i\kappa x\cdot \theta},  
\end{equation}
\begin{equation}
\frac{\partial U_{0}^{s}}{\partial |x|}-i\kappa U_{0}^{s}=o\left(\frac{1}{|x|}\right), |x|\rightarrow\infty. \label{radiationc-B-1}
\end{equation}
 
\item Assume in addition  
that $\lambda_0$ and  $K\mid_{\Omega}$ are in $C^{0, \gamma}(\Omega)$, 
$\gamma \in (0, 1]$ and 
the reference bodies $B_m$'s have the same perimeter and diameter, and 
denote by 
$P:=\frac{\vert \partial B \vert}{diam (B)}$. Then
\begin{equation}\label{C}
  U^\infty(\hat{x},\theta)= U_0^\infty(\hat{x},\theta) +
  O(a^{\min\{\gamma,\; \frac{2-\beta}{3},\; 1-3\beta,\; 2- \beta-t\}}) 
 \end{equation}
 uniformly in terms of $\theta$ and $\hat{x}$ in  $\mathbb{S}^2$, 
 where $\bold{P}_0=P$ in $\Omega$ and $P_0=0$ in $\mathbb{R}^{3} \setminus{\overline \Omega}$.
\end{enumerate}

\end{theorem}
 \bigskip
 
 We see from case 2 of Theorem \ref{equivalent-medium} that
 \begin{equation}\label{C-1}
  U^\infty(\hat{x},\theta)= U_0^\infty(\hat{x},\theta)+\left\{\begin{array}{ccc}
O(a^{\min\{\gamma,\; \frac{2-\beta}{3},\; 2- \beta-t\}})&   \mbox{ if } 0\leq \beta \leq \frac{1}{8}\\
O(a^{\min\{\gamma,\; 1-3\beta,\; 2- \beta-t\}})& \mbox{ if } \frac{1}{8} \leq \beta \leq \frac{1}{3}
\end{array}\right.
 \end{equation}
and for $\beta \geq \frac{1}{3}$, the remainder is no longer tending to zero as $a$ tends to zero.

We also see that the best error estimate is attained for $\beta=0$ and 
it is $O(a^{\min\{\gamma,\; \frac{2}{3},\; 2-t\}})$.
In particular if we reasonably take $t \leq 1$, which means that the minimum distance 
is of the order of the diameters, i.e. $d \approx a$, and $\gamma \geq \frac{2}{3}$, then 
 \begin{equation}\label{C-2}
  U^\infty(\hat{x},\theta)= U_0^\infty(\hat{x},\theta)+
O(a^{\frac{2}{3}}),\; a \longrightarrow 0.
 \end{equation}
 However in this case, i.e. $\beta =0$, the number of small holes attains its maximum, 
 $M=O(a^{\beta-2})=O(a^{-2})$. Actually, there is a compromise between the number of the small holes and 
 the order of the approximation. In short, the larger is the number of the small holes (or the smaller is 
 $\beta$) the better 
 is the approximation.
 
 \subsection{Some applications}\label{applications}
 As a corollary of Theorem \ref{equivalent-medium}, we deduce the following results.
 \begin{enumerate}
 \item We write $\kappa^2 n^2 +(K+1)P_0 \lambda_0 =\kappa^2\tilde{n}^2$, i.e. 
 $\tilde{n}^2=n^2 +\frac{(K+1)P_0 \lambda_0}{\kappa^2}$. 
 This way of representing the equivalent coefficient 
 $\kappa^2 n^2 +(K+1)P_0 \lambda_0$ means that the equivalent material behaves 
 as an acoustic material
 whose index of refraction is $\tilde{n}$ satisfying 
 $
  \tilde{n}^2=n^2 +\frac{(K+1)P_0 \lambda_0}{\kappa^2}.
 $
 In particular, we can choose $\lambda_0$, see remark \ref{choosing-the-impedance}, as
 \begin{equation}\label{choice of impedance}
 \lambda_0:=\tilde{\lambda}_0\; \kappa^2;
 \end{equation}
 then
 \begin{equation}\label{assumotion?}
  \tilde{n}^2=n^2 +(K+1)P_0 \tilde{\lambda}_0.
 \end{equation}
 We set $\tilde{n}=\tilde{n}_1+i \tilde{n}_2$. This new acoustic material 
 will be passive only if $\Im\; \tilde{n}=\tilde{n}_2 \geq 0$. 
 From (\ref{assumotion?}), we deduce that
 \begin{equation}\label{beta-1-beta-2}
 \tilde{n}_1^2-\tilde{n}_2^2=n^2\; +\; 2\pi (K+1)P_0 \Re\; 
 \tilde{\lambda}_0 ~~\mbox{ and }~~ \; \tilde{n}_1\; \tilde{n}_2=
 \pi(K+1)P_0 \Im\; \tilde{\lambda}_0.
 \end{equation}
 Recall that the coefficient $\lambda_0$ comes from the surface impedance 
 functions $\lambda_{m}$ attached to every small body of the collection of 
 the small bodies
 generating the coefficient $n^2 +(K+1)P_0 \tilde{\lambda}_0$. Hence, if we choose 
 $\lambda_m$'s so that $\Im \tilde{\lambda}_0>0$, then necessarily 
 $\tilde{n}_1>0$ and then we can generate acoustic materials having index of 
 refraction as
 \begin{equation}\label{usual-index}
\Re \tilde{n} >0\;~ \mbox{ and }~ \Im \tilde{n}>0. 
 \end{equation} 
 Now, if we choose $\lambda_m$'s so that $\Im(\tilde{\lambda}_0)<0$, 
 then we deduce from (\ref{beta-1-beta-2}) and the fact that $\tilde{n}_2 \geq 0$ that necessarily 
 $\tilde{n}_1<0$. 
 With this way, we can generate acoustic materials of the form:
 \begin{equation}\label{Unsual-index}
 \Re\; \tilde{n} <0\; \mbox{ and } \; \Im\; \tilde{n} >0.
 \end{equation}
 In addition, if we choose the surface impedance function so that 
 $\lambda_0:=\lambda_0(\epsilon)$ with $\Re\; \lambda_0(\epsilon)\; (\; >0 \;) \rightarrow 0$ and 
 $\Im\; \lambda_0(\epsilon) (\; <0\; ) \rightarrow 0$, as $\epsilon \rightarrow 0$, so that
 the condition (\ref{condition-lambda+and-lambda-}) is satisfied i.e. 
 $\frac{\Re\; \lambda_0(\epsilon)}{\vert \lambda_0(\epsilon)\vert^2}>\frac{\sqrt{26 M_{max}}}{\pi }$
 , then the two equations in (\ref{beta-1-beta-2}) imply that
 $$
 \Re\; \tilde{n}(\epsilon)\rightarrow -n\; \mbox{ and }\; \Im\; \tilde{n}(\epsilon)\rightarrow 0\;~  \mbox{ as } \epsilon \rightarrow 0.
 $$
 
 As a conclusion, if we perforate a given acoustic material, modeled by the index of 
 refraction $n(x)$, with appropriately distributed small holes with
 negative imaginary part surface impedance functions, attached to each surface, 
 the equivalent medium behaves as a passive acoustic medium only if it is an acoustic 
 metamaterial, i.e. with index of refraction
\begin{equation}\label{Pure-metamaterial} 
 \tilde{n}(x)=-n(x).
\end{equation} 
 
 \item We can choose the surface impedance functions so that 
 $n^2 +\frac{(K+1)P_0 \lambda_0}{\kappa^2}=1,\; x\in \Omega$. For instance we take 
 $\lambda_0:=\tilde{\lambda}_0\; \kappa^2$ and 
 $\tilde{\lambda}_0:=\frac{1-n^2}{(K+1)P_0}$ . Hence the equivalent index of refraction $\tilde{n}$ 
 is $\tilde{n}(x)=1,\; x \in \mathbb{R}^3$. This means that the region $\Omega$ modeled 
 by the index of refraction $n$ is approximately cloaked. Observe that since we can take the surface impedances complexe
 valued with any sign then we can cloak the region $\Omega$ defined by complex valued index of 
 refraction with any sign of the real and imaginary parts. 
 
 \end{enumerate}
 
 \begin{remark}\label{choosing-the-impedance}
  The surface impedance in (\ref{choice of impedance}), i.e.
  $\lambda_0:=\tilde{\lambda}_0\; \kappa^2$, will be achieved if we choose the surface impedances of the 
  small holes as 
  \begin{equation}\label{frequency-dependent-surface-impedance}
  \lambda_m:=\lambda_m(\kappa):=\tilde{\lambda}_{m,0}\; 
  \kappa^2 a^{-\beta},
  \end{equation}
  with $a <<1$, for instance.  Since these surface impedance functions appear in the boundary conditions (\ref{acgoverningsupport-1}), i.e. 
  \begin{equation}
\frac{\partial U_m}{\partial \nu_m}+\lambda_m(\kappa) U_m=0,\, 
1\leq m \leq M,\; \mbox{ on } \partial D_m, \nonumber 
\end{equation}
 which we can rewrite for convenience, to link the acoustic pressure $U_m$ to the 
 velocity on the boundary 
 $\frac{\partial U_m}{\partial \nu_m}$, as
 \begin{equation}
U_m + \sigma_m(\kappa)\; \frac{\partial U_m}{\partial \nu_m}=0,\, 
1\leq m \leq M, \; \mbox{ on } \partial D_m \label{acgoverningsupport-1--2}  
\end{equation}
where $\sigma_m(\kappa):=\lambda_m^{-1}(\kappa)$.  In the time domain these impedance boundary conditions are translated as
\begin{equation}
 \tilde{U}_m(t, x)+ \int^\infty_{-\infty} \tilde{\sigma}_m(t-t') \frac{\partial \tilde{U}_m}
 {\partial \nu_m} (t', x)dt'=0,\; x \mbox{ on } \partial D_m. \label{acgoverningsupport-1--3}  
\end{equation}
 where $\tilde{\sigma}_m(t):=\frac{1}{2\pi}\int^\infty_{-\infty} \sigma(\kappa)
 e^{-i\kappa t} d\kappa$
  and $\tilde{U}_m(t,x):=\frac{1}{2\pi}\int^\infty_{-\infty} \tilde{U}_m(\kappa, x)
  e^{-i\kappa t} d\kappa$. 
 \bigskip
 
 The wave propagation with the type of time domain impedance boundary conditions in 
 (\ref{acgoverningsupport-1--3}) 
 were recently object of studies see for instance \cite{D-BB:2014}. Sufficient conditions on the 
 admissibility of the impedance boundary conditions (\ref{acgoverningsupport-1--3}), as the reality, passivity and 
 causality conditions,  are given in the book  \cite{R-H:2015} and also discussed in 
 \cite{D-BB:2014}. These conditions are given in the frequency domain and, if $\Re\; 
 \tilde{\lambda}_0 \geq 0$ and due to the 
 decay in terms of $\kappa$,  our surface 
 impedances $\sigma_m(\kappa)=\lambda^{-1}_m$ with $\lambda_m$ given in 
 (\ref{frequency-dependent-surface-impedance}) satisfy those conditions. 
 As a consequence, we do hope that the choice of the surface impedances used in the applications described above 
 might make sense in practice.
 \end{remark}

\section{Proof of Theorem \ref{equivalent-medium}}\label{Proof-Theorem} 
%  The idea of the proof follows \cite{C-S:2015} and its improvement in \cite{F-D-M:2016JOE}. 
%  We describe the main steps of the proof and provide the needed changes to 
%  perform. 
\subsection{The relative distribution of the small bodies}\label{The-relative-distribution-of-the-small-bodies}
We start with the following observation from \cite{F-D-M:2016JOE} on the relative distribution of the 
small bodies. For $m=1,\dots, M$ fixed, we distinguish between the obstacles $D_j$, $j\neq\,m$, by keeping them into different layers based on their distance from $D_m$.  
Let us first assume that $K(z_m)=0$ for every $z_m$. Hence each $\Omega_m$ has the (same) volume $a^{2-\beta}$ and contains only one obstacle $D_m$. 
 We arrange these cubes in cuboids, in different layers such that the total cubes upto the $n^{th}$
layer consists of $(2n+1)^3$ cubes for $n=0,\dots,[\left(a^{\frac{2-\beta}{3}}-\frac{a}{2}\right)^{-1}]$, 
and $\Omega_m$ is located at the center, see Fig \ref{distribution-obstacles}. 
Hence the number of obstacles, we denote by $ D^n_j$ and located in the $n^{th}$, $n\neq0$, layer  
will be $[(2n+1)^3-(2n-1)^3]=24 n^2+2$ and their distance from $D_m$ is 
greater than ${n}\left(a^{\frac{2-\beta}{3}}-\frac{a}{2}\right)$. Observe that, $\frac{2-\beta}{2}
a^{\frac{2-\beta}{3}}\leq \left(a^{\frac{2-\beta}{3}}-\frac{a}{2}\right)\leq\;a^{\frac{2-\beta}{3}}$.
Hence we deduce the needed estimate
\begin{equation}\label{distance-between-layers}
 d(D^n_j, D_m) \geq \frac{n a^{\frac{2-\beta}{3}}}{2}.
\end{equation}
Now, we come back to the case where $K(z_m)\neq 0$. As 
$\frac{1}{2}\leq \frac{[K(z_m)+1]}{K(z_m)+1}\leq 1$, then with such $\Omega_m$'s, 
the total cubes located in the $n^{th}$ layer consists of at most the double 
of $[(2n+1)^3-(2n-1)^3]$, i.e. $48n^2+4$ and the inequality (\ref{distance-between-layers}) is also verified.

% \textcolor{red}{ Sir, here, I did not change the size from $(\frac{a}{2}+d^\alpha)$ to $(\frac{a}{2}+a^\alpha)$, as it seems the sufficient condition for the Neumann series looks 
% better when it involves $d$. Precisely, $\epsilon^{2-\beta}d^{-3\alpha}<c$ is better than $\epsilon^{2-\beta-3\alpha}<c$. So I feel we need to think in a different direction.} \\
\begin{figure}[htp]
\centering
\includegraphics[width=4.5cm,height =4.5cm]{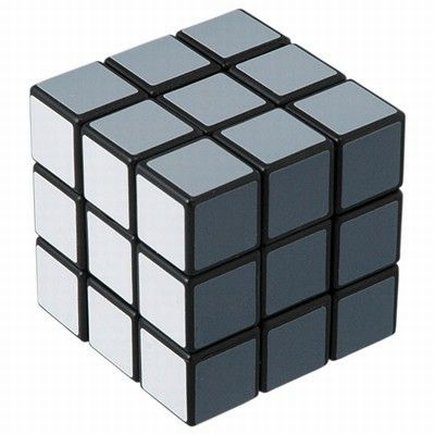}
\caption{Rubik's cube consisting of two layers}\label{fig:1-acsmall}
\end{figure}

\subsection{Solvability of the linear-algebraic system \eqref{fracqcfracmain-1}}\label{Solvability-of-the-linear-algebraic-system-elastic-small}

The algebraic system 
\eqref{fracqcfracmain-1} can be written in compact form as
\begin{equation}\label{compacfrm1}
 \mathbf{B}{Q}=\mathrm{U}^I,
\end{equation}
\noindent
where ${Q},\mathrm{U}^I \in \mathbb{C}^{M\times 1}\mbox{ and } \mathbf{B}\in\mathbb{C}^{M\times M}$ are defined as;
\begin{eqnarray}
\mathbf{B}:=\left(\begin{array}{ccccc}
   -\frac{1}{{C}_1} &-G_\kappa(z_1,z_2)&-G_\kappa(z_1,z_3)&\cdots&-G_\kappa(z_1,z_M)\\
-G_\kappa(z_2,z_1)&-\frac{1}{{C}_2}&-G_\kappa(z_2,z_3)&\cdots&-G_\kappa(z_2,z_M)\\
 \cdots&\cdots&\cdots&\cdots&\cdots\\
-G_\kappa(z_M,z_1)&-G_\kappa(z_M,z_2)&\cdots&-G_\kappa(z_M,z_{M-1}) &-\frac{1}{{C}_M}
   \end{array}\right),\label{mainmatrix-acoustic-small}\\
\nonumber\\
%\begin{split}
{Q}:=\left(\begin{array}{cccc}
    {Q}_1 & {Q}_2 & \ldots  & {Q}_M
   \end{array}\right)^\top \text{ and } 
\mathrm{U}^I:=\left(\begin{array}{cccc}
     V_n^{t}(z_1) & V_n^{t}(z_2)& \ldots &  V_n^{t}(z_M)
   \end{array}\right)^\top.\label{coefficient-and-incidentvectors-acoustic-small}
%\end{split}
\end{eqnarray}
The following lemma
provides us with the needed estimate on the invertibility of (\ref{compacfrm1}).
\begin{lemma}\label{Mazyawrkthmac}%[Based on work of Maz'ya]
We distinguish the following two cases:
\begin{itemize}
 \item Let $\Re (\lambda_{m,0}) < 0$ and assume that 
 $\min\limits_{1\leq{m}\leq{M}}\frac{\Re{C}_m}{\vert{C}_m\vert^2}>\frac{\sqrt{2 M_{max}}}{\pi\,a^{2-\beta}}$ 
 then the matrix $\mathbf{B}$ 
is invertible and the solution vector ${Q}$ of \eqref{compacfrm1} satisfies the estimate
\begin{equation}\label{mazya-fnlinvert-small-ac-2}
\begin{split}
 \sum_{m=1}^{M}|{Q}_m|^{2}
\leq4 \left(\frac{\min\limits_{1\leq{m}\leq{M}}
\Re{C}_m}{\max\limits_{1\leq{m}\leq{M}}\vert{C}_m\vert^2}-\frac{\sqrt{26 M_{max}}}{\pi\,a^{2-\beta}}
\right)^{-2}\sum_{m=1}^{M}\left|V_n^{t}(z_m)\right|^2
\end{split}
\end{equation}
and then 
\begin{equation}\label{mazya-fnlinvert-l-1-norm}
\begin{split}
 \sum_{m=1}^{M}|{Q}_m|
\leq 2 \left(\frac{\min\limits_{1\leq{m}\leq{M}}
\Re{C}_m}{\max\limits_{1\leq{m}\leq{M}}\vert{C}_m\vert}-\frac{\max\limits_{1\leq{m}\leq{M}}\vert{C}_m\vert\sqrt{26 M_{max}}}{\pi\,a^{2-\beta}}
\right)^{-1} M \max\limits_{1\leq{m}\leq{M}}\vert{C}_m\vert \sum_{m=1}^{M}\left|V_n^{t}(z_m)\right|,
\end{split}
\end{equation}
\item Let $\Re (\lambda_{m,0}) > 0$ and assume that $\frac{\min\limits_{1\leq{m}\leq{M}}\Re(-{C}_m)}{(\max\limits_{1\leq{m}\leq{M}}\vert{C}_m\vert)^2}>\frac{\sqrt{2M_{max}}}{\pi a^{2-\beta}}$ then
 the matrix $\mathbf{B}$ 
is invertible and the solution vector ${Q}$ of \eqref{compacfrm1} satisfies the estimate
\begin{equation}\label{mazya-fnlinvert-small-ac-2-1}
 \begin{split}
 \sum_{m=1}^{M}|{Q}_m|^{2}
\leq 4\left(\frac{\min\limits_{1\leq{m}\leq{M}}\Re(-{C}_m)}{\max\limits_{1\leq{m}\leq{M}}\vert{C}_m\vert^2}-\frac{\sqrt{2 M_{max}}}{\pi\,a^{2-\beta}}\right)^{-2}\sum_{m=1}^{M}\left|V_n^{t}(z_m)\right|^2.
 \end{split}
 \end{equation}
 and then
 \begin{equation}\label{mazya-fnlinvert-l-1-norm-2}
\begin{split}
 \sum_{m=1}^{M}|{Q}_m|
\leq 2 \left(\frac{\min\limits_{1\leq{m}\leq{M}}
\Re{-C}_m}{\max\limits_{1\leq{m}\leq{M}}\vert{C}_m\vert}-\frac{\max\limits_{1\leq{m}\leq{M}}\vert{C}_m\vert\sqrt{26 M_{max}}}{\pi\,a^{2-\beta}}
\right)^{-1} M \max\limits_{1\leq{m}\leq{M}}\vert{C}_m\vert \sum_{m=1}^{M}\left|V_n^{t}(z_m)\right|,
\end{split}
\end{equation}
 \end{itemize}
\end{lemma}

The proof of this lemma is given in \cite{C-S:2016} for the case where $n=1$, i.e. 
$G_\kappa$ is the fundamental solution $\Phi_\kappa(x, y):=\frac{e^{i \kappa \vert x-y \vert}}{4\pi \vert x-y\vert}$ 
of the free Helmholtz model. But that proof goes smoothly for $G_\kappa$ as well.

Since ${C}_m:=-\lambda_m \vert \partial D_m \vert$, 
the condition $\frac{\min\limits_{1\leq{m}\leq{M}}\vert\Re{C}_m\vert}{\max\limits_{1\leq{m}\leq{M}}\vert{C}_m\vert^2}>\frac{\sqrt{26M_{max}}}{\pi\,a^{2-\beta}}$ 
is satisfied if  $\lambda_-$ and $\lambda_+$ satisfy 
\begin{equation}\label{condition-lambda+and-lambda-}
\frac{\lambda_-}{  \lambda_+^2}>\frac{\sqrt{26 M_{max}}}{\pi }.
\end{equation}

\subsection{The limiting model}\label{subsection-piecewise-constant} 
 
From the function $K$, we define a bounded function $K_a: \mathbb{R}^3\rightarrow \mathbb{R}$ as follows: 
\begin{equation}
K_a(x):= K_a(z_m):=
\left\{\begin{array}{ccc}
K(z_m)+1&   \mbox{ if }& x\in \Omega_m, \\
0 & \mbox{ if }& x\notin \Omega_m \mbox{ for any } m=1,\dots,[a^{\beta-2}].
\end{array}\right.
\end{equation}
Hence each $\Omega_m$ contains $[ K_a(z_m)]$ obstacles and $K_{max}:=\sup_{z_m}K_a(z_m)$.

 \par
Let  $\bold{C}_a$ be the piecewise constant functions such that 
$\bold{C}_a\vert_{\Omega_m}=\bar{C}_m :=\lambda_{m,0}\frac{\vert\partial{B_m}\vert}{\left(\max\limits_{m} diam (B_m)\right)^2}$ for all $m=1,\dots,M$ 
and vanishes outside $\Omega$. The constants $\bar{C}_m$ are independent of $a$. We set 
\begin{equation}\label{sup-cap-all}
\mathcal{C}:=\max\limits_{1\leq{m}\leq{M}}\vert \bar{C}_m\vert_{\infty}.
\end{equation}
Consider the Lippmann-Schwinger equation
\begin{eqnarray}\label{fracqcfracmain-effect-int}
 U_a(z) +\int_{\Omega} G_\kappa(z,y)K_a(y){\bold{C}_a}(y) U_a(y) dy &=&-
 V_n^{t}(z, \theta), z\in \Omega
\end{eqnarray}
and set the Poisson potential
\begin{eqnarray}
 V(Y)(x):=\int_{\Omega}G_\kappa(x,y)K_a(y){\bold{C}_a}(y)Y(y)dy,\qquad x\in\mathbb{R}^3.
\end{eqnarray}

The coefficients $K_a$ and $\bold{C}_a$ are uniformly bounded. 
The next lemma concerns the mapping properties of the Poisson potential. 

\begin{lemma}\label{Lip-Schw}
The operator $V:{L}^2(\Omega)\rightarrow {H}^2(\varOmega)$ is well defined and it is a bounded operator for any bounded domain $\varOmega$ in $\mathbb{R}^3$, 
i.e. there exists a positive 
constant $c_0$ such that
\begin{eqnarray}\label{H2normofY}
\|V(Y)\|_{H^{2}(\Omega)}\leq c_0 \|Y\|_{L^{2}(\Omega)}.
%\mbox{ and }& \|Y\|_{H^{2}(\Omega_m)}\leq c_0 \|Y\|_{L^{2}(\Omega_m)}
\end{eqnarray}

\end{lemma}

Its proof is given in\cite{C-K:1998}, for instance, in the case when $G_\kappa(x, y)=\Phi_\kappa(x, y)
:=\frac{e^{i \kappa \vert x-y \vert}}{4\pi \vert x-y\vert}$. However $G_\kappa -\Phi_\kappa$ satisfies 
$(\Delta +\kappa^2)(G_\kappa -\Phi_\kappa)=\kappa^2(1-n^2)\Phi_\kappa$ in $\mathbb{R}^3$. By interior
estimates, we deduce that $\Vert G_\kappa(\cdot, z) -\Phi_\kappa(\cdot, z)\Vert_{H^2(\Omega)}$ is 
uniformly bounded in terms of $z$. By this last property, we show that Lemma \ref{Lip-Schw} 
is also valid for $n\neq 1$ and $n=1$ in $\mathbb{R}^3 \setminus \Omega$.

%\begin{proofe} {\it{Proof of Lemma \ref{Lip-Schw}}}
 %....
%\end{proofe}

\bigskip

Using Lemma \ref{Lip-Schw}, the fact that the operator $I+V:L^2(\Omega) \longmapsto L^2(\Omega)$ is Fredholm with zero index 
and the uniqueness of the scattering problem corresponding to the model 
\begin{equation}\label{elas-scat-equi}
 (\Delta +\kappa^2 n^2-K_a \bold{C}_a) Y =0,\; \mbox{ in } \mathbb{R}^3
\end{equation}
(where $Y:=Y^i +Y^s$ and $Y^s$ satisfies the Sommerfeld radiation conditions and $Y^i$ is an 
incident field),
we have the following lemma, see \cite{C-S:2015} the details. 

 \bigskip
 
\begin{lemma}\label{invertibility-of-VC}
 There exists one and only one solution $Y$ of the Lippmann-Schwinger equation 
 (\ref{fracqcfracmain-effect-int}) and it satisfies the estimate
 \begin{eqnarray}\label{est-Lipm-Sch}
  \Vert Y\Vert_{L^\infty(\Omega)}\leq C \Vert V_n^{t}\Vert_{H^2(\Omega)}& \mbox{ and }& \Vert \nabla Y\Vert_{L^\infty(\Omega)}\leq C^{\prime} \Vert V_n^{t}\Vert_{H^2(\tilde{\Omega})},
 \end{eqnarray}
where $\tilde{\Omega}$ being a large bounded domain which contains $\bar{\Omega}$.
\end{lemma}

\subsubsection{Case when the obstacles are arbitrarily distributed}\label{arbitrarely-distributed}

From the definition of $\bold{C}_a$, we have $\bold{C}_a\mid_{\Omega_m}=
\lambda_0(z_m) \frac{\vert B_m\vert}
{\max\{ diam (B_m)^2\}}$. We have assumed the reference 
 bodies $B_m$'s to have different but uniformly bounded perimeters 
 with uniformly lower bounded radiis. 
In addition, since $\lambda_0$ is a continuous function, the function $\lambda_a: \Omega \longrightarrow \mathbb{C}
$ defined $\lambda_a\mid_{\Omega_m}:=\lambda(z_m)$, converges to $\lambda_0$
uniformly. Then there exists a function $\bold{P}_0$
 in $L^2(\Omega)$ such that $\bold{C}_a$ converges weakly to 
 $\bold{C}_0:=\lambda_0 \bold{P_0}$
 in $L^{2}(\Omega)$. Now, since $K$ is continuous hence $K_a$ converges to $(K+1)$ in 
 $L^\infty(\Omega)$ and hence in $L^2(\Omega)$. 
 Then we can show that $K_a \bold{C}_a$ converges to 
 $(K+1) \bold{C}_0$ in $L^2(\Omega)$.  
 \bigskip
 
 \noindent Since $K \bold{C}_a$ is bounded in $L^{\infty}(\Omega)$, then from the invertibility of the 
 Lippmann-Schwinger equation and the mapping properties of the Poisson potential, 
 see Lemma \ref{invertibility-of-VC},
 we deduce that $\Vert U_{a}^{t}\Vert_{H^2(\Omega)}$ is bounded and in particular, up to a sub-sequence, $U_{a}^{t}$ tends to $U_{0}^{t}$ in $L^2(\Omega)$. 
 From the convergence of $K_a \bold{C}_a$ to $(K+1)\bold{C}_0$
 and the one of $U_{a}^{t}$ to $U_{0}^{t}$ and (\ref{fracqcfracmain-effect-int}), 
 we derive the following equation satisfied by 
 $U_{0}^{t}(x)$
 $$
 U_{0}^{t}(x) +\int_{\Omega}G_\kappa(x, y) (K_a)(y)\bold{C}_0(y) U_{0}^{t}(y)dy=-V_n^{t}( x, \theta)\; \mbox{ in } \Omega.
 $$
This is the Lippmann-Schwinger equation corresponding to the scattering problem
$(\Delta + \kappa^{2}-(K+1) \bold{C}_0)U_{0}^{t}=0 \mbox{ in }\mathbb{R}^{3}\label{piecewise-2}$,
$U_{0}^{t}=U_{0}^s +U^i$,
and $U^s$ satisfies the Sommerfeld radiation conditions. As the corresponding far-fields are of the form
 $$
 U_{0}^{\infty}(\hat{x}, \theta)=\int_{\Omega}G^\infty_\kappa(\hat{x}, y)(K+1)(y)\bold{C}_0(y)U_{0}^{t}(y)dy
 $$
 which can be written, by the mixed reciprocity relation $G^\infty_\kappa(\hat{x}, y)=
 V_n(y, -\hat{x})$, as
 
 $$
 U_{0}^{\infty}(\hat{x}, \theta)=\int_{\Omega} V_n(y, -\hat{x})(K+1)(y)\bold{C}_0(y)U_{0}^{t}(y)dy
 $$
 and similarly the ones of $U^t_{a}$ are of the form
 $$
 U_{a}^{\infty}(\hat{x}, \theta)=\int_{\Omega} V_n(y, -\hat{x})K_a(y)\bold{C}_a(y)U_{a}^{t}(y)dy
 $$
 we deduce that
 $$
 U^\infty_{a}(\hat{x},\theta)-U^\infty_{0}(\hat{x},\theta)=o(1),\; a\rightarrow 0, \mbox{ uniformly in terms of } \hat x , \theta\; \in \mathbb{S}^{2}.
 $$

 \subsubsection{Case when $K$ is H$\ddot{\mbox{o}}$lder continuous}\label{smoothly-distributed}
 
 If we assume that $K\in C^{0, \gamma}(\Omega),\; \gamma \in (0, 1]$, 
 then we have the estimate $\Vert (K+1) -K_a\Vert_{L^\infty(\Omega)}\leq
 C a^{\gamma}$, $a<<1$. Similarly, since we assume the shapes to have the same perimeter and diameter and
 $\lambda_0$ to be in $C^{0, \gamma}(\Omega),\; \gamma \in (0, 1]$, then we have also 
 $\Vert \bold{C}_0 -\bold{C}_a\Vert_{L^\infty(\Omega)}\leq
 C a^{\gamma}$ since $\bold{C}_0=\lambda_0\; \frac{\vert \partial B \vert}{diam(B)}$ where 
 $B$ is the common reference body. 
 Since the obstacles have the same perimeter, we set $\bold{C}_0$ to be a constant in 
 $\Omega$ and $\bold{C}_0=0$ in $\mathbb{R}^3\setminus \Omega$. 
 Recall that $U_0$ and $U_a$ are solutions of the Lippmann-Schwinger equations 
 $$
 U_0+\int_{\Omega}G_\kappa(x, y)(K+1)(y)\bold{C}_0(y)U_{0}^{t}(y)dy=V_n^{t}
 $$
 and 
 $$
 U_a+\int_{\Omega}G_\kappa(x, y) K_a(y)\bold{C}_0(y)U_{a}^{t}(y)dy=V_n^{t}.
 $$
 From the estimate $\Vert (K+1) -K_a\Vert_{L^\infty(\Omega)}\leq C a^{\gamma}$, $a<<1$, we derive the estimate
 \begin{equation}\label{appro-0-a-1}
  U_0^\infty(\hat{x}, \theta)-U^\infty_a(\hat{x}, \theta)=O(a^\gamma),\; a<<1, \mbox{ uniformly in terms of } \hat x , \theta\; \in \mathbb{S}^{2}.
 \end{equation}

\subsection{The approximation by the algebraic system}\label{the approximation-by-AS}
For each $m=1,\dots,M$, we rewrite the equation \eqref{fracqcfracmain-effect-int} as follows
\begin{eqnarray}\label{fracqcfracmain-effect-int-1}
 U_a(z_m) &+&\sum_{\substack{j=1 \\ j\neq m}}^{M} G_\kappa(z_m,z_j)\bar{C}_j U_a(z_j) a^{2-\beta}
 \nonumber\\
 &=&-V_n^{t}(z_m, \theta)
 +\sum_{\substack{j=1 \\ j\neq m}}^{M} G_\kappa(z_m,z_j)\bar{C}_j U_a(z_j) a^{2-\beta}-\sum_{\substack{j=1 \\ j\neq m}}^{[a^{\beta-2}]} G_\kappa(z_m,z_j)K_a(z_j)\bar{C}_j U_a(z_j) Vol(\Omega_j)\nonumber\\
 &&+\sum_{\substack{j=1 \\ j\neq m}}^{[a^{\beta-2}]} G_\kappa(z_m,z_j)K_a(z_j)
 \bar{C}_j U_a(z_j) Vol(\Omega_j)-\int_{\Omega} G_\kappa(z_m,y)K_a(y)
 {\bold{C}_a}(y) U_a(y) dy.
\end{eqnarray}

Let us estimate the following quantities:
$$
A:=\sum_{\substack{j=1 \\ j\neq m}}^{[a^{\beta-2}]} G_\kappa(z_m,z_j)K_a(z_j)\bar{C}_j U_a(z_j) Vol(\Omega_j)-\int_{\Omega} G_\kappa(z_m,y)K_a(y)
{\bold{C}_a}(y) U_a(y) dy
$$
and 
$$
B:= \sum_{\substack{j=1 \\ j\neq m}}^{M} G_\kappa(z_m,z_j)\bar{C}_j U_a(z_j) a^{2-\beta}-\sum_{\substack{j=1 \\ j\neq m}}^{[a^{\beta-2}]} G_\kappa(z_m,z_j)K_a(z_j)\bar{C}_j U_a(z_j) Vol(\Omega_j).
$$

\subsubsection{Estimate of $A$}
By the decomposition of $\Omega$, $\Omega:=\cup^{[a^{\beta-2}]}_{l=1}\Omega_l$, we have
 \begin{equation}\label{integralonomega}
            \int_{\Omega} G_\kappa(z_m,y){K_a}(y){\bold{C}_a}(y) U_a(y) dy=\sum_{l=1}^{[a^{\beta-2}]}\int_{\Omega_l} G_\kappa(z_m,y){K_a}(y){\bold{C}_a}(y) U_a(y) dy.
       \end{equation}
%  hence       
\begin{align}
\mbox{Hence,\quad }A&:= \int_{\Omega_m} G_\kappa(z_m,y)K_a(y){\bold{C}_a}(y) U_a(y) dy\nonumber\\
&\quad+\sum_{\substack{j=1 \\ j\neq m}}^{[a^{\beta-2}]} \left [G_\kappa(z_m,z_j)K_a(z_j)\bar{C}_j U_a(z_j) Vol(\Omega_j)
                  -\int_{\Omega_j} G_\kappa(z_m,y)K_a(y){\bold{C}_a}(y) U_a(y) dy \right].\label{rearranging-A}
\end{align}
 
 For $l\neq m$, we have
 \begin{align}\label{integralonomega-subelements}
  \int_{\Omega_l} G_\kappa(z_m,y){K_a}(y){\bold{C}_a}(y) U_a(y) dy &- G_\kappa(z_m,z_l){K_a}(z_l)\bar{C}_l U_a(z_l) Vol(\Omega_l) \nonumber\\
  &={K_a}(z_l)\bar{C}_l\int_{\Omega_l} \left[G_\kappa(z_m,y) U_a(y)-G_\kappa(z_m,z_l) U_a(z_l)\right] dy.
 \end{align}
  
We set $f(z_m,y)=G_\kappa(z_m,y)  U_a(y)$ then $f(z_m,y)$ satisfies
  $$f(z_m,y)-f(z_m,z_l)=(y-z_l)R^i(z_m,y)$$
  where
  \begin{eqnarray}\label{taylorremind1}
   R^i(z_m,y)
   &=&\int_0^1\nabla_y f(z_m,y-\beta(y-z_l))\,d\beta\nonumber\\
   &=&\int_0^1\nabla_y \left[G_\kappa(z_m,y-\beta(y-z_l)) U_{a}(y-\beta(y-z_l))\right]\,d\beta\nonumber\\
   &=&\int_0^1\left[\nabla_yG_\kappa(z_m,y-\beta(y-z_l))\right] U_{a}(y-\beta(y-z_l))\,d\beta\nonumber\\
   &&+\int_0^1G_\kappa(z_m,y-\beta(y-z_l))\left[\nabla_y U_{a}(y-\beta(y-z_l))\right]\,d\beta.
  \end{eqnarray}

 We set $\Phi_{\kappa, j}(x, y):=\frac{e^{i\kappa n(z_j)\vert x-y\vert}}{4\pi \vert x-y\vert}$. 
 We have the following lemma:
\begin{lemma}\label{asymp-G-lem} We have the asymptotic expansion:
\begin{equation}\label{asymp-G}
\partial_{x_i}^\alpha G_\kappa (z_j, y)= \partial^\alpha_{x_i} \Phi_{\kappa, j}(z_j, y)+
O(1),\; \mbox{ for }  y \mbox{ in } \Omega, \mbox{ with } \alpha=0,  1.
\end{equation}
\end{lemma}

\begin{proof}
We know that $(\Delta +\kappa^2 n^2(x))G_{\kappa}=-\delta(\cdot - y),$ in $\mathbb{R}^3$ and 
$\Phi_{\kappa, j}(x, y):=\frac{e^{i\kappa n(z_j)\vert x-y\vert}}{4\pi \vert x-y\vert}$ 
satisfies $(\Delta +\kappa^2 n^2(z_j))\Phi_{\kappa}=-\delta(\cdot-y),$ in $\mathbb{R}^3$, 
where both $G_\kappa$ and $\Phi_{\kappa, j}$ satisfy the Sommerfeld radiation condition. 
Then $H_{\kappa, j}(x, z):=(G_{\kappa}-\Phi_{\kappa,j})(x, z)$ satisfies the same radiation 
condition and 
\begin{equation}\label{difference}
(\Delta +\kappa^2 n^2(z_j))H_{\kappa, j}=\kappa^2 (n^2(z_j)-n^2(x))G_\kappa, \mbox{ in } \mathbb{R}^3. 
\end{equation}
Multiplying both sides of (\ref{difference}) by $\Phi_{\kappa, j}$ and integrating over $B \supset \supset
\Omega$ we obtain:
\begin{eqnarray}\label{representation-1}
H_{\kappa, j}(x, y)&=&\kappa^2\int_{B}(n^2(z_j)-n^2(t))\Phi_{\kappa, j}(t, x)G_{\kappa, j}(t, y)dt \\%\nonumber\\
&& +\int_{\partial B} H_{\kappa, j}(t, y)\partial_{\nu(t)}\Phi_{\kappa, j}(t, x)ds(t) 
+\int_{\partial \Omega} \partial_{\nu(t)} H_{\kappa, j}(t, x)\Phi_{\kappa, j}(t, x)ds(t),\nonumber%\\%\quad
\end{eqnarray}
and then 
\begin{eqnarray}\label{representation-2}
\nabla_{x} H_{\kappa, j}(x, y)&=&\kappa^2\int_{B}(n^2(z_j)-n^2(t))G_\kappa(t, y)\nabla_{x}
\Phi_{\kappa, j}(t, x)dt \\%\nonumber\\
&&+\int_{\partial B} H_{\kappa, j}(t, y)\nabla_{x}\partial_{\nu(t)}\Phi_{\kappa, j}(t, x)ds(t)  
+ \int_{\partial B} \partial_{\nu(t)} H_{\kappa, j}(t, y)\nabla_x \Phi_{\kappa, j}(t, x)ds(t).\nonumber
\end{eqnarray}
Since $(n^2(z_j)-n^2(t))\nabla_x \Phi_\kappa(t, z_j)=O(\vert t-z_j\vert), t \in B$, 
as the singularity of $\nabla_{x}\Phi_{\kappa, j}(t, x)$ is of the order $\vert t-x\vert^{-2}$, 
then both integrals appearing in (\ref{representation-1}) and (\ref{representation-2}) are of the 
order $O(1)$ for $ x=z_j $ and $ y \in \Omega \subset \subset B$.
\end{proof}  
 
 %   From (\ref{kupradzeten1}) and (\ref{gradkupradzeten1}), we derive for $l \neq m$
From the explicit form of $\Phi_{\kappa, m}$, we have $\nabla_y\Phi_{\kappa, m}(x,y)=
\Phi_{\kappa, m}(x,y)\left[\frac{1}{|x-y|}-i\kappa n(z_m)\right]\frac{x-y}{|x-y|}, {x}\neq{y}$.  
Now from Section \ref{The-relative-distribution-of-the-small-bodies}, precisely the inequality (\ref{distance-between-layers}), we see that for $l \neq m$
   \begin{eqnarray*}
   \vert\Phi_{\kappa, m}(z_m,y-\beta(y-z_l))\vert\leq\frac{\mathring{c}}{4\pi\; 
   n\frac{a^{\frac{2-\beta}{3}}}{2}},& \mbox{ and }&\vert\nabla_y\Phi_{\kappa, m}
  (z_m,y-\beta(y-z_l))\vert\leq\frac{\mathring{c}}{4\pi\;n^2 \left(\frac{a^{\frac{2-\beta}{3}}}{2}
  \right)^{2}}
  \end{eqnarray*}
  where $\mathring{c}$ depends only on $\kappa$ and $n(z_m)$. Combining these estimates with 
  \eqref{asymp-G} of Lemma \ref{asymp-G-lem}, we derive the inequalities
  \begin{align}\label{upper-bound-G-kappa}
   \vert G_{\kappa}(z_m,y-\beta(y-z_l))\vert\leq\frac{\mathring{c}}{4\pi\; 
   n\frac{a^{\frac{2-\beta}{3}}}{2}}, \mbox{ and }&\vert\nabla_y G_{\kappa}(z_m,y-\beta(y-z_l))\vert\leq\frac{\mathring{c}}{4\pi\;n^2 \left(\frac{a^{\frac{2-\beta}{3}}}{2}\right)^{2}}.
  \end{align}
  
% \textcolor{black}{Here we can write $\frac{a^{\frac{1}{3}}}{2}$ instead of $\frac{d}{T}$ in the above.} 
 Then,
 \begin{align}\label{taylorremind-effect}
 \vert R_l(z_m,y) \vert
   \leq & \frac{\mathring{c}}{2\pi\;n a^{\frac{2-\beta}{3}}}\left(\frac{{1}}{na^{\frac{2-\beta}{3}}}\int_0^1 {\vert U_{a}(y-\beta(y-z_l))\vert}d\beta+\int_0^1{\vert\nabla_y U_{a}(y-\beta(y-z_l))\vert}d\beta\right).
  \end{align}

  Then, for $l\neq m$, \eqref{integralonomega-subelements} and \eqref{taylorremind-effect} and observing that $\bar{C}_l$ is a constant in $\Omega_l$, imply the estimate
  
% % % \begin{eqnarray}\label{integralonomega-subelements-abs}
% % %   \Big\vert\int_{\Omega_l}   &    G_\kappa(z_m,y) {K_a}(y){\bold{C}_a}(y) U_a(y) dy-G_\kappa(z_m,z_l) & {K_a}(z_l)\bar{C}_l U_a(z_l) Vol(\Omega_l)\Big\vert 
% % %   \nonumber \\
% % %   \leq\,  &\frac{\mathring{c} \bar{C}_l K_a(z_l)}{\pi\;n^2 a^{\frac{2(2-\beta)}{3}}}\int_{\Omega_l} \Big[\int_0^1 \vert U_{a}  (y-\beta(y-z_l))\vert\;d\beta\Big] &     \vert y-z_l\vert dy
% % %     \nonumber\\
% % %                &&
% % %               + \frac{\mathring{c} \bar{C}_l K_a(z_l)}{2\pi\;n a^{\frac{2-\beta}{3}}}   \int_{\Omega_l} \left[\int_0^1 {\vert \nabla_y U_{a}(y-\beta(y-z_l))\vert}d\beta \right]\vert y-z_l\vert dy 
% % %  \nonumber\\ 
% % %     \substack{\leq\\ \eqref{est-Lipm-Sch}}&c_1\frac{[{K_a}(z_l)]\bar{C}_l}{n^2 a^{\frac{2(2-\beta)}{3}}}\,a^\frac{4(2-\beta)}{3}
% % %                     \,\substack{\leq\\ \eqref{sup-cap-all}}\, c_1\frac{K_{\max}\mathcal{C}}{n^2 }\,a^\frac{2(2-\beta)}{3},\qquad
% % %  \end{eqnarray} 
\begin{align}\label{integralonomega-subelements-abs}
  \Big\vert\int_{\Omega_l}       G_\kappa(z_m,y) {K_a}(y){\bold{C}_a}(y) U_a(y) dy&-G_\kappa(z_m,z_l)  {K_a}(z_l)\bar{C}_l U_a(z_l) Vol(\Omega_l)\Big\vert 
  \nonumber \\
  \leq\,  \frac{\mathring{c} \bar{C}_l K_a(z_l)}{\pi\;n^2 a^{\frac{2(2-\beta)}{3}}}\int_{\Omega_l} \Big[\int_0^1 \vert U_{a} & (y-\beta(y-z_l))\vert\;d\beta\Big] \vert y-z_l\vert dy
    \nonumber\\
       &+ \frac{\mathring{c} \bar{C}_l K_a(z_l)}{2\pi\;n a^{\frac{2-\beta}{3}}}   \int_{\Omega_l} \left[\int_0^1 {\vert \nabla_y U_{a}(y-\beta(y-z_l))\vert}d\beta \right]\vert y-z_l\vert dy 
 \nonumber\\ 
    \substack{\leq\\ \eqref{est-Lipm-Sch}}c_1\frac{[{K_a}(z_l)]\bar{C}_l}{n^2 a^{\frac{2(2-\beta)}{3}}}\,a^\frac{4(2-\beta)}{3}
                    \,\substack{\leq\\ \eqref{sup-cap-all}}\,& c_1\frac{K_{\max}\mathcal{C}}{n^2 }\,a^\frac{2(2-\beta)}{3},\qquad
 \end{align}
 \noindent
for a suitable constant $c_1$.

 Regarding the integral $\int_{\Omega_m} G_\kappa(z_m,y){\bold{C}_a}(y) U_a(y) dy$ we do the following estimates:
\begin{eqnarray}\label{estmatemthint-effe-acc}
 \left\vert\int_{\Omega_m} G_\kappa(z_m,y){K_a}(y){\bold{C}_a}(y) U_a(y) dy\right\vert
 \leq \frac{3}{8\pi}c_1K_{max}\mathcal{C}\left(\frac{4}{3\pi} \right)^\frac{1}{3}a^\frac{2(2-\beta)}{3}.
\end{eqnarray}

From \eqref{rearranging-A}, we can have
\begin{align*}
\vert A\vert \leq&\, \vert \int_{\Omega_m} G_\kappa(z_m,y)K_a(y){\bold{C}_a}(y) U_a(y) dy \vert \\
 &\quad+\sum_{\substack{j=1 \\ j\neq m}}^{[a^{\beta-2}]} \left[ \vert G_\kappa(z_m,z_j)K_a(z_j)
 \bar{C}_j U_a(z_j) Vol(\Omega_j)-\int_{\Omega_j} G_\kappa(z_m,y)K_a(y){\bold{C}_a}(y) U_a(y) dy \vert 
 \right]
\end{align*}
which we can estimate by
\begin{align*}
\vert A\vert \leq 
\sum^{[2a^{\frac{\beta-2}{3}}]}_{n=1}&2[(2n+1)^3-(2n-1)^3]
\bigg[ \vert G_\kappa(z_m,z_j)K_a(z_j)\bar{C}_j U_a(z_j) Vol(\Omega_j)\vert \\
&-\int_{\Omega_j} \vert G_\kappa(z_m,y)K_a(y){\bold{C}_a}(y) U_a(y)\vert dy \bigg]+\vert 
\int_{\Omega_m} G_\kappa(z_m,y)K_a(y){\bold{C}_a}(y) U_a(y) dy \vert.
\end{align*}
and then
$$
\vert A\vert \leq CK_{max}[a^{\frac{2(2-\beta)}{3}}+a^{\frac{2-\beta}{3}}].
$$
Finally 
$$
\vert A \vert \leq CK_{max} a^{\frac{2-\beta}{3}}.
$$
\subsubsection{Estimate of $B$}
\begin{equation}
\begin{split}
\sum_{\substack{j=1 \\ j\neq m}}^{M} G_\kappa(z_m,z_j)\bar{C}_j U_a(z_j)a^{2-\beta} -\sum_{\substack{j=1 \\ j\neq m}}^{[a^{\beta-2}]}
        & G_\kappa(z_m,z_j)K_a(z_j)\bar{C}_j U_a(z_j) Vol(\Omega_j)           
\\ \nonumber
=\sum_{\substack{l=1 \\ l\neq m\\ z_l \in \Omega_m}}^{[K_a(z_m)]}G_\kappa(z_m,z_l)\bar{C}_l U_a(z_l) a^{2-\beta}+&
\sum_{\substack{j=1 \\ j\neq m}}^{[a^{\beta-2}]} \sum_{\substack{l=1 \\ z_l \in \Omega_j}}^{[K_a(z_j)]}G_\kappa(z_m,z_l)\bar{C}_l U_a(z_l) a^{2-\beta}
\nonumber\\-\sum_{\substack{j=1 \\ j\neq m}}^{[a^{\beta-2}]} G_\kappa(z_m,z_j)&K_a(z_j)\bar{C}_j U_a(z_j) Vol(\Omega_j)              \\
=\bar{C}_m a^{2-\beta}\sum_{\substack{l=1 \\ l\neq m\\ z_l \in \Omega_m}}^{[K_a(z_m)]}G_\kappa(z_m,z_l) U_a(z_l)  
\nonumber \\
+\sum_{\substack{j=1 \\ j\neq m}}^{[a^{\beta-2}]}
\bar{C}_j a^{2-\beta}\big[\big(&\sum_{\substack{l=1 \\ z_l \in \Omega_j}}^{[K_a(z_j)]}G_\kappa(z_m,z_l) U_a(z_l)\big)-
 G_\kappa(z_m,z_j)[K_a(z_j)] U_a(z_j)\big],
 \end{split}
 \end{equation}
 since $ Vol(\Omega_j)=a^{2-\beta}\frac{[K_a(z_j)]}{K_a(z_j)}\; \mbox{ and } 
 \bar{C}_l=\bar{C}_j, \mbox{ for } l=1, ...,\; K_a(z_j)$. We write,

\begin{eqnarray}\label{Ej1}
E^j_1&:=&\sum_{\substack{l=1 \\ l\neq m\\ z_l \in \Omega_m}}^{[K_a(z_m)]}G_\kappa(z_m,z_l) U_a(z_l)\\
\label{Ej2}
\mbox{and  }\qquad\qquad E^j_2&:=&\big[\big(\sum_{\substack{l=1 \\ z_l \in \Omega_j}}^{[K_a(z_j)]}G_\kappa(z_m,z_l) U_a(z_l)\big)-
 G_\kappa(z_m,z_j)[K_a(z_j)] U_a(z_j)\big]\nonumber\\
 &=&\sum_{\substack{l=1 \\ z_l \in \Omega_j}}^{[K_a(z_j)]}\big(G_\kappa(z_m,z_l) U_a(z_l)-
 G_\kappa(z_m,z_j) U_a(z_j)\big).
\end{eqnarray}

We need to estimate $\bar{C}_m a^{2-\beta} E^j_1$ and $\sum_{\substack{j=1 \\ j\neq m}}^{[a^{\beta-2}]}\bar{C}_j a^{2-\beta}E^j_2$. \\
% \textcolor{black}{For $\Omega_m,\; m=1,..., [a^{-1}]$, we maybe need to seperate the cases: near $\Omega_j$'s and far $\Omega_j$'s!!!!}\newline

\bigskip
Now by writing $f'(z_m,y):=G_\kappa(z_m,y) U_a(y)$. For $z_l\in\Omega_j,\,j\neq m$, using Taylor series, we can write
  $$f'(z_m,z_j)-f'(z_m,z_l)=(z_j-z_l)R'(z_m;z_j,z_l),$$
  with
 \begin{eqnarray}\label{taylorremind1'}
   R'(z_m;z_j,z_l)
   &=&\int_0^1\nabla_y f'(z_m,z_j-\beta(z_j-z_l))\,d\beta.
  \end{eqnarray}
By doing the computations similar to the ones we have performed in (\ref{taylorremind1}-\ref{taylorremind-effect}) and by using Lemma \ref{invertibility-of-VC} and Lemma \ref{asymp-G-lem}, we obtain
\begin{eqnarray}
  \vert \sum_{\substack{j=1 \\ j\neq m}}^{[a^{\beta-2}]}\bar{C}_j a^{2-\beta} E^j_2 \vert \leq c_2\mathcal{C} K_{max} a^{\frac{2-\beta}{3}}   \label{este2j1}
\end{eqnarray}
One can easily see that,
\begin{eqnarray}
\vert \bar{C}_m a^{2-\beta} E^j_1\vert \leq \frac{c_1(K_{max}-1)\mathcal{C}}{4\pi}\frac{a^{2-\beta}}{d} = \frac{c_1(K_{max}-1)\mathcal{C}}{4\pi}
a^{2-\beta-t}.   \label{este1j}
\end{eqnarray}
\bigskip

Substitution of \eqref{integralonomega} in \eqref{fracqcfracmain-effect-int-1} and using the estimates 
\eqref{integralonomega-subelements-abs} and \eqref{estmatemthint-effe-acc}  associated to $A$ and the estimates 
\eqref{este2j1} and \eqref{este1j} associated to $B$  gives us

\begin{eqnarray}%\label{fracqcfracmain-effect-int-2}
 U_a(z_m) +\sum_{\substack{j=1 \\ j\neq m}}^{M} G_\kappa(z_m,z_j)\bar{C}_j U_a(z_j)
 a^{2-\beta}
  &=&-V_n^{t}(z_m, \theta)\\
  &&+O\left(c_2\mathcal{C} K_{max} a^{\frac{2-\beta}{3}}\right)+O\left(\frac{c_1(K_{max}-1)\mathcal{C} }{4\pi} a^{2-\beta-t}\right).\nonumber\label{fracqcfracmain-effect-int-3}
\end{eqnarray}

We rewrite the algebraic system (\ref{fracqcfracmain-1}) as 
\begin{equation}\label{alge-rewriten-U}
U_{a,m} + \sum_{\substack{j=1 \\ j\neq m}}^{M} G_\kappa(z_m,z_j)\bar{C}_j U_{a, j} a^{2-\beta}
  =-V_n^{t}(z_m)
\end{equation} 
where we set $U_{a,m}:=C^{-1}_m Q_m$, recalling that {$C_m=\bar{C}_m\; a^{2-\beta}$}.
\bigskip

Taking the difference between \eqref{fracqcfracmain-effect-int-3} and 
\eqref{alge-rewriten-U}  produces the algebraic system
\begin{eqnarray}\label{fracqcfracmain-effect-int-4}
 (U_{a,m}-U_a(z_m)) +\sum_{\substack{j=1 \\ j\neq m}}^{M} G_\kappa(z_m,z_j)\bar{C}_j (U_{a,j}-U_a(z_j)) a^{2-\beta} 
 &=&O\left(\mathcal{C} K_{max}(a^{\frac{2-\beta}{3}}+a^{2-\beta- t}) \right).\nonumber
\end{eqnarray}

Comparing this system with \eqref{fracqcfracmain-1} and by using Lemma \ref{Mazyawrkthmac}, we obtain the estimate

\begin{eqnarray}\label{mazya-fnlinvert-small-ac-3-effect-dif}
 \sum_{m=1}^{M}(U_{a,m}-U_a(z_m))&=&O\left(\mathcal{C} K_{max} M (a^{\frac{2-\beta}{3}}+a^{2-\beta- t}) \right).
\end{eqnarray}

For the special case $d=a^t,\,M=O(a^{\beta-2})$ with  $t>0$, we have the following approximation of the far-field from 
the Foldy-Lax asymptotic expansion \eqref{x oustdie1 D_m farmain-1} and from the definitions 
$U_{a,m}:=C^{-1}_m Q_m$ and $C_m:=-\lambda_m\vert\partial D_m\vert=\bar{C}_m a^{2-\beta}$, for $m=1,\dots,M$:
\begin{eqnarray}\label{x oustdie1 D_m farmain-recent**-effect}
4\pi\;U^\infty(\hat{x},\theta) &=&V_n^{\infty}+\sum_{j=1}^{M}V_n(z_j, -\hat{x})\bar{C}_jU_{a,j}\; a^{2-\beta}+O\left(a^{3-s-2\beta}\right)\nonumber\\
&=&V_n^{\infty}+\sum_{j=1}^{M}V_n(z_j, -\hat{x})\bar{C}_jU_{a,j}\; a^{2-\beta}+O\left(a^{1-3\beta}\right).
% \\ \nonumber&&\textcolor{black}{\left[O\left(a+a^{2\alpha}+a^{1-\alpha}+a^{\frac{1+2\alpha}{3}}\right)\quad \mbox{for }t=\frac{1}{3}, s=1.\right]} 
\end{eqnarray}
Consider the far-field of type:
\begin{eqnarray}\label{acoustic-farfield-effect}
U^\infty_{\bold{C}_a}(\hat{x},\theta) &=& V_n^{\infty}+\frac{1}{4\pi}\int_{\Omega} V_n(y, -\hat{x}){K_a}(y){\bold{C}_a}(y) U_a(y) dy. \nonumber
\end{eqnarray}
corresponding to the scattering problem (\ref{elas-scat-equi}).
Taking the difference between \eqref{acoustic-farfield-effect} and \eqref{x oustdie1 D_m farmain-recent**-effect} we have:
 \begin{align}\label{acoustic-difference-farfield-effect}
 4\pi& (U^\infty_{\bold{C}_a}(\hat{x},\theta)-U^\infty(\hat{x},\theta)) \nonumber\\
 =& \int_{\Omega} V_n(y, -\hat{x}){K_a}(y){\bold{C}_a}(y) U_a(y) dy- \sum_{j=1}^{M}V_n(z_j, -\hat{x})\bar{C}_jU_{a,j}a^{2-\beta}
 +O\left(a^{1-3\beta}\right)\nonumber
\\ \nonumber =& \sum_{j=1}^{[a^{\beta-2}]}\int_{\Omega_j} V_n(y, -\hat{x}){K_a}(y){\bold{C}_a}(y) U_a(y)dy 
  -\sum_{j=1}^{[a^{\beta-2}]} \sum_{\substack{l=1 \\ z_l \in \Omega_j}}^{[K_a(z_j)]}V_n(z_l, -\hat{x})\bar{C}_lU_{a,l} a^{2-\beta}
+O\left(a^{1-3\beta}\right)
\\ \nonumber =& \sum_{j=1}^{[a^{\beta-2}]}{K_a}(z_j)\bar{C}_j\int_{\Omega_j} \left[V_n(y, -\hat{x}) U_a(y) - V_n(z_j, -\hat{x})U_a(z_j)\right]dy \nonumber
\\ \nonumber &+ \sum_{j=1}^{[a^{\beta-2}]}\bar{C}_ja^{2-\beta} \left[\sum_{\substack{l=1 \\ z_l \in \Omega_j}}^{[K_a(z_j)]}\left(V_n(z_j, -\hat{x}) U_a(z_j)-V_n(z_l, -\hat{x})U_a(z_l)\right)+\sum_{\substack{l=1 \\ z_l \in \Omega_j}}^{
[K_a(z_j)]} V_n(z_l, -\hat{x}) \left(U_a(z_l)-U_{a,l}\right)\right] \nonumber
\\ \nonumber &+O\left(a^{1-3\beta}\right)
 \\ \nonumber =& \sum_{j=1}^{[a^{\beta-2}]}\int_{\Omega_j}{K_a}(z_j)\bar{C}_j \left[V_n(y, -\hat{x}) U_a(y) - V_n(z_j, -\hat{x})U_a(z_j)\right]dy\nonumber\\
 &+ \sum_{j=1}^{[a^{\beta-2}]}\bar{C}_ja^{2-\beta} \sum_{\substack{l=1 \\ z_l \in \Omega_j}}^{
 [K_a(z_j)]}\left(V_n(z_j, -\hat{x}) U_a(z_j)-V_n(z_l, -\hat{x})U_a(z_l)\right)+\sum_{j=1}^{M}V_n(z_j, -\hat{x})\bar{C}_ja^{2-\beta} \left[U_a(z_j)-U_{a,j}\right]\nonumber\\ \nonumber
 \\ \nonumber &+O\left(a^{1-3\beta}\right)
\\ \nonumber \substack{= \\ \eqref{mazya-fnlinvert-small-ac-3-effect-dif} }& \sum_{j=1}^{[a^{\beta-2}]}{K_a}(z_j)\bar{C}_j\int_{\Omega_j} \left[V_n(y, -\hat{x}) U_a(y) - V_n(z_j, -\hat{x})U_a(z_j)\right]dy
\\ \nonumber
&+ \sum_{j=1}^{[a^{\beta-2}]}\bar{C}_ja^{2-\beta} \sum_{\substack{l=1 \\ z_l \in \Omega_j}}^{
 [K_a(z_j)]}\left(V_n(z_j, -\hat{x}) U_a(z_j)-V_n(z_l, -\hat{x})U_a(z_l)\right)+O\left(\mathcal{C}^2 K_{max}  (a^{\frac{2-\beta}{3}}+a^{2-\beta- t}) \right)\nonumber
\\ &+O\left(a^{1-3\beta}\right).
 \end{align}
Now, let us estimate the  difference $\sum_{j=1}^{[a^{\beta-2}]}{K_a}(z_j)\bar{C}_j\int_{\Omega_j} \left[V_n(y, -\hat{x}) U_a(y) - 
V_n(z_j, -\hat{x})U_a(z_j)\right]dy$.
Write, $f_1(y)=V_n(y, -\hat{x}) U_a(y)$. Using Taylor series, we can write
  $$f_1(y)-f_1(z_j)=(y-z_j)\cdot R_j(y),$$
  with 
  \begin{eqnarray}
   R_j(y)
   &=&\int_0^1\nabla_y (f_1)(y-\beta(y-z_j))\,d\beta\nonumber\\
   &=&\int_0^1\left[\nabla_y \left[V_n(-\hat{x},y-\beta(y-z_j)) U_{a}(y-\beta(y-z_j))\right] \right]\,d\beta\nonumber\\
   &=&\int_0^1\left[\nabla_y V_n(-\hat{x},y-\beta(y-z_j))\right] U_{a}(y-\beta(y-z_j))\,d\beta\nonumber\\
%    &&+\int_0^1e^{-i\kappa\hat{x}\cdot(y-\beta(y-z_j))}\left[\nabla_yC(y-\beta(y-z_j))\right] Y(y-\beta(y-z_j))\,d\beta\nonumber\\
   &&+\int_0^1V_n(-\hat{x},y-\beta(y-z_j))\left[\nabla_y U_{a}(y-\beta(y-z_j))\right]\,d\beta.
  \end{eqnarray}

  Recall that $V_n$ satisfies the scattering problem (\ref{acimpoenetrable-2})-
  (\ref{radiationc-2}), hence it is also solution of the corresponding 
  Lippmann-Schwinger equation $V_n(x) +\int_\Omega \Phi(x, y)(n^2(y)-1)V_n(y) dy=-
  e^{\kappa x\cdot \theta}$. This is the same integral equation 
  as (\ref{fracqcfracmain-effect-int}) at the expense of replacing
  $K_a C_a$ by $n-1$ and $V_n$ by
  $e^{\kappa x\cdot \theta}$. Then replacing in Lemma \ref{invertibility-of-VC} $V_n$ by
  $e^{\kappa x\cdot \theta}$ and then $Y$ by $V_n$, we have the
  following estimates 
  \begin{equation}
   \Vert V_n\Vert_{L^\infty(\Omega)}, \Vert \nabla V_n \Vert_{L^\infty(\Omega)}
   \leq \tilde{C}
  \end{equation}
where $\tilde{C}$ depends only on $\Vert n\Vert_{L^\infty(\Omega)}$ and $\kappa$. Then
  
 \begin{eqnarray}\label{taylorremind-effect-singvariable}
 \vert R_j(y) \vert
   &\leq& \tilde{C} \left(\int_0^1 |U_{a}(y-\beta(y-z_j))|\, d\beta\,+\,
   \int_0^1\vert\nabla_y U_{a}(y-\beta(y-z_j))\vert \,d\beta\right).
  \end{eqnarray}
  Using  \eqref{taylorremind-effect-singvariable} we get the estimate
% % %   \newpage
% % %  \begin{equation}\label{integralonomega-subelements-abs-effect-----------}
% % %   \left\vert\sum_{j=1}^{[a^{\beta-2}]}{K_a}(z_j)\bar{C}_j\int_{\Omega_j} 
% % %   \left[V_n(y, -\hat{x})(y) U_a(y) - V_n(z_j, -\hat{x})U_a(z_j)\right]dy\right
% % %   \vert 
% % % \end{equation}
% % % 
% % % $$
% % % \leq \tilde{C}\quad\sum_{j=1}^{[a^{\beta-2}]}{K_a}(z_j)\bar{C}_j \left(\kappa 
% % % \int_{\Omega_j}\vert y-z_j\vert \int_0^1 |U_{a}(y-\beta(y-z_j))|\, d\beta\,dy\right)\, 
% % % $$
% % % 
% % % $$
% % %   \quad+\,\tilde{C} \sum_{j=1}^{[a^{\beta-2}]}{K_a}(z_j)\bar{C}_j 
% % %   \left(\int_{\Omega_j}\vert y-z_j\vert\int_0^1\vert\nabla_y U_{a}(y-\beta(y-z_j))\vert \,d\beta\,dy\right)
% % %   $$
% % %   
% % %   $$ 
% % %   \leq \quad \tilde{C} \sum_{j=1}^{[a^{\beta-2}]}{K_a}(z_j)\bar{C}_jc_1\, a^{2-\beta}\,a^\frac{2-\beta}{3}\,\left(\kappa +c_5\right) 
% % %    \leq \tilde{C} K_{max} \mathcal{C}c_1\left(\kappa +c_5\right)\,a^{\frac{2-\beta}{3}}.   
% % % $$

 \begin{align}\label{integralonomega-subelements-abs-effect}
  \bigg\vert\sum_{j=1}^{[a^{\beta-2}]}{K_a}(z_j)\bar{C}_j\int_{\Omega_j} &
  \big[V_n(y, -\hat{x})(y) U_a(y) - V_n(z_j, -\hat{x})U_a(z_j)\big]dy\bigg\vert 
  \nonumber\\
\leq \tilde{C}\sum_{j=1}^{[a^{\beta-2}]}{K_a}(z_j)\bar{C}_j& \bigg(\kappa 
\int_{\Omega_j}\vert y-z_j\vert \int_0^1 |U_{a}(y-\beta(y-z_j))|\, d\beta\,dy\bigg)\, 
\nonumber\\
 +\tilde{C} \sum_{j=1}^{[a^{\beta-2}]}&{K_a}(z_j)\bar{C}_j
  \left(\int_{\Omega_j}\vert y-z_j\vert\int_0^1\vert\nabla_y U_{a}(y-\beta(y-z_j))\vert \,d\beta\,dy\right)
\nonumber\\
  \leq  \tilde{C} \sum_{j=1}^{[a^{\beta-2}]}{K_a}(z_j)\bar{C}_j&c_1\, a^{2-\beta}\,a^\frac{2-\beta}{3}\,\left(\kappa +c_5\right) 
  \nonumber\\
   \quad\qquad\leq \tilde{C} K_{max} \mathcal{C}c_1\big(\kappa +c_5&\big)\,a^{\frac{2-\beta}{3}}.  
\end{align}

 In the similar way, using (\ref{mazya-fnlinvert-small-ac-3-effect-dif}),  we have,
 \begin{align}\label{integralonomega-subelements-abs-effect-1}
 \left\vert \sum_{j=1}^{[a^{\beta-2}]}\bar{C}_ja^{2-\beta} \sum_{\substack{l=1 \\ 
 z_l \in \Omega_j}}^{[K_a(z_j)]}\left(V_n(z_j, -\hat{x}) U_a(z_j)-
 V_n(z_l, -\hat{x})U_a(z_l)\right)\right\vert
 &\leq O\left(K_{max} \mathcal{C}(a^{\frac{2-\beta}{3}}+a^{2-\beta-t})\right).
 \end{align}
 
 Using the estimates \eqref{integralonomega-subelements-abs-effect} and \eqref{integralonomega-subelements-abs-effect-1} in \eqref{acoustic-difference-farfield-effect}, we obtain
 
 \begin{align}\label{acoustic-difference-farfield-effect-1}
%  \begin{split}
 \frac{1}{4\pi}&U^\infty_{\bold{C}_a}(\hat{x},\theta)-U^\infty(\hat{x},\theta) \nonumber\\
 =&\, O\left(K_{max}a^\frac{2-\beta}{3} \mathcal{C}c_1\left(\kappa +c_5\right)\right) + 
 O(\mathcal{C}(\mathcal{C}+1)K_{max} M\; a^{2-\beta}(a^{\frac{2-\beta}{3}}+a^{2-\beta-t}))
   +O\left(a^{1-3\beta}\right)
 \nonumber\\
    =&\,O\left( a^{\frac{2-\beta}{3}}+a^{2-\beta-t}+a^{1-3\beta}\right).
%     \end{split}
  \end{align}
%Remember that we have the freedom in choosing $\beta$ as soon as it %satisfies $3\beta t \geq 1$. We take for simplicity $\beta$ so that %$3\beta t \geq 1$ and then $\frac{4}{3}-3 \beta t=\frac{1}{3}$.  
 Since $Vol(\Omega)$ is of order $a^{\beta-2}(\frac{a^{2-\beta}}{2}+\frac{d}{2})^3$, and $d$ is 
 of the order $a^t$, we should have {$t\geq\frac{2-\beta}{3}$}. Hence, we need to impose the following 
 conditions
\begin{eqnarray*}
  t\geq\frac{2-\beta}{3},\;  2-\beta-t \geq 0\;    
      \mbox{ and }\; 1-3\beta>0.
  \end{eqnarray*}  

 \subsection{End of the proof of Theorem \ref{equivalent-medium}}
  Combining the estimates (\ref{acoustic-difference-farfield-effect-1}) and (\ref{appro-0-a-1}), we deduce that
\begin{equation}\label{final}
  \frac{1}{4\pi} \left[U^\infty(\hat{x}, \theta)-U_{0}^\infty(\hat{x}, \theta) 
  \right ]\cdot \hat{x}=O(a^{\min \{\gamma,\; \frac{1}{3},\; \frac{2-\beta}{3},\; 2-\beta-t,\; 
  1-3\beta\}}),\; a<<1, \;~~ 
  \frac{2-\beta}{3} \leq t \leq 2-\beta
 \end{equation}
uniformly in terms of  $\hat x , \theta\; \in \mathbb{S}^{2}$.

 \begin{center} {\bf{Acknowledgment}} \end{center}
This work was funded by the Deanship of Scientific Research (DRS), King Abdulaziz University, under the grant no. 20-130-36-HiCi. The authors, therefore, acknowledge with thanks DRS technical and financial support.

\bibliographystyle{abbrv}
%  \bibliography{/home/dchalla/Paperwork2/Elastic-Pointlike/References_papr/allreferences}

\end{document}